\documentclass[10pt,a4paper,oneside]{amsproc}

\usepackage{latexsym, amsthm}
\usepackage{amsfonts, amsmath, amssymb}
\usepackage{euscript,mathrsfs}
\usepackage{url}
\usepackage{array}
\usepackage[a4paper]{geometry}
\usepackage{graphics}
\usepackage[usenames]{color}
\usepackage[all]{xy}
\usepackage{graphicx}
\usepackage{boxedminipage}

\newtheorem{theorem}{Theorem}[section]
\newtheorem{proposition}[theorem]{Proposition}
\newtheorem{lemma}[theorem]{Lemma}
\newtheorem{corollary}[theorem]{Corollary}

\newtheorem{remark}[theorem]{Remark}
\newtheorem{conjecture}[theorem]{Conjecture}
\newtheorem{definition}[theorem]{Definition}
\newtheorem{question}[theorem]{Question}

\newcommand{\LL}{\mathcal {L}}
\newcommand{\J}{\mathscr {J}}
\newcommand{\F}{\mathbb {F}}
\newcommand{\im}{\mathrm {im}}
\newcommand{\B}{\mathcal {B}}
\newcommand{\C}{\mathcal {C}}
\newcommand{\Fqt}{\mathbb{F}_{q^2}}
\newcommand{\Fq}{\mathbb{F}_q}
\def\PP{{\mathbb P}}

\def\Fq{{\mathbb F}_q}

\def\FF{{\mathbb F}}
\def\PP{{\mathbb P}}
\newcommand{\V}{{\mathsf{V}}}

\newcommand{\X}{\mathcal{X}}
\newcommand{\T}{\mathcal{T}}

\title[A proof of S{\o}rensen's conjecture on Hermitian surfaces]{A proof of S{\o}rensen's conjecture on Hermitian surfaces}
\author{Peter Beelen}
\address{Department of Applied Mathematics and Computer Science, \newline \indent
Technical University of Denmark, DK 2800, Kgs. Lyngby, Denmark}
\email{pabe@dtu.dk}

\author{Mrinmoy Datta}
\address{Institute of Mathematics and Statistics, \newline \indent
University of Troms\o, 9037 Troms\o, Norway}
\email{mrinmoy.dat@gmail.com}
\author{Masaaki Homma}
\address{Department of Mathematics and Physics \newline \indent Kanagawa University, Hiratsuka 259-1293, Japan}
\email{homma@kanagawa-u.ac.jp}


\date{}
\begin{document}

\begin{abstract}
In this article we prove a conjecture formulated by A.B.~S\o rensen in 1991 on the maximal number of $\Fqt$-rational points on the intersection of a non-degenerate Hermitian surface and a surface of degree $d \le q.$
\end{abstract}

\maketitle

\section{Introduction}

Algebraic varieties $V$ defined over a finite field $\mathbb{F}_q$ with $q$ elements, occur in the interplay of various mathematical disciplines:
algebraic geometry, finite geometry, coding theory, to name a few.
%The function field of a variety is formed by the set of functions defined on a dense part of the variety.
To study the behaviour of algebraic functions of $V$, one often would like to count the number of $\F_q$-rational points where it vanishes.
In particular, for a given polynomial $F$, one would like to compute the cardinality of the set of $\Fqt$-rational points in the intersection
$V(F) \cap V,$ where $V(F)$ denotes the variety defined by the equation $F=0.$
Hermitian varieties defined over a finite field $\mathbb{F}_{q^2}$ are particularly well-studied, since they were introduced in 1966 by Bose and
Chakravarti, \cite{BC,C}.
In particular, the line-plane incidence with respect to the non-degenerate Hermitian surfaces is well understood, see \cite{CK, GK, IZ2} among others.
Also from the point of view of algebraic error-correcting codes Hermitian varieties have been studied, since they have a large number of
$\mathbb{F}_{q^2}$-rational points, see for example \cite[Examples 6.5, 6.6]{L}.
The application to coding theory leads to the following question arises, which was formulated by A.B. S\o rensen in \cite{SoT}:

\begin{question}\label{q} \normalfont
Let $F \in \Fqt[x_0, x_1, x_2, x_3]$ be a homogeneous polynomial of degree $d$ and $V_2$ denote a non-degenerate Hermitian surface in $\PP^3$ defined over $\Fqt$.
What is the maximum number of $\Fqt$-rational points in $V (F) \cap V_2$?
\end{question}

An answer to this question, would determine the minimum distance of the codes coming from Hermitian surfaces, but is also of
independent interest.
For example, an equivalent question is to ask for the maximum number of rational points a hyperplane section can have with the $d$-uple embedding
of the Hermitian surface.
In this context, Chakravarti \cite{Ch} studied the $2$-uple embedding of the cubic surface defined by the equation $x_0^3 + x_1^3 + x_2^3 + x_3^3 = 0$ in
$\PP^3(\mathbb{F}_4)$.
S{\o}rensen's conjecture from 1991 is the following:

\begin{conjecture}\cite[Page 9]{SoT}\label{conj:soerensen}
Let $F$ and $V_2$ be as in Question \ref{q}. If $d \le q$, then
$$|(V(F) \cap V_2)(\Fqt)| \le d (q^3 + q^2 - q) + q + 1.$$
Further, the surfaces given by a homogeneous polynomial  $F \in \Fqt[x_0, x_1, x_2, x_3]$ attaining the above upper bound are given by a union of
$d$ planes in $\PP^3 (\Fqt)$ that are tangent to $V_2$, each containing a common line $\ell$ intersecting $V_2$ at $q+1$ points.
\end{conjecture}

Since $V_2(\Fqt)=(q^3+1)(q^2+1)$, see \cite{BC}, it is obvious that $|(V(F) \cap V_2)(\Fqt)| \le (q^3+1)(q^2+1).$
Moreover, if $d \ge q+1$, then one can find a homogeneous polynomial $F \in \Fqt[x_0, x_1, x_2, x_3]$ of degree $d$ such that
$|V(F)(\Fqt) \cap V_2| = (q^3+1)(q^2+1).$
Indeed, one can simply choose $F$ to be a suitable multiple of the defining equation of $V_2$.
This makes the degree restriction $d \le q$ in Conjecture \ref{conj:soerensen} quite natural.
On the other hand, one could modify Question \ref{q} by asking for the maximum number of $\Fqt$-rational points in $V (F) \cap V_2$
under the condition that $V_2$ is not contained in $V(F)$. We will make this assumption in the remainder of the paper.

Currently, Conjecture \ref{conj:soerensen} is only known to hold for $d \in \{1,2,3\}.$
For $d=1$ it follows from results stated in \cite{BC}, see Theorem \ref{linear} below for more details.
For $d=2$, the conjecture was proven in 2007 \cite{E}, while recently the conjecture was proven to be true for $d=3$ in \cite{BD}.
The techniques used to prove the conjecture in \cite{E} and \cite{BD} were of a geometrical nature, and do not generalize to higher degree in an obvious way.
In this article, using a mix of geometric and combinatorial arguments, we prove Conjecture \ref{conj:soerensen}.
A byproduct of our methods results in an answer to the modified question for $d=q+1$.
More precisely, we will show that Conjecture \ref{conj:soerensen} also holds for $d=q+1$, provided the maximum is taken over all homogeneous
polynomials that are not a multiple of the defining equation of $V_2.$

\section{Preliminaries}\label{prel}

In the remainder of this paper, $q$ will denote a fixed prime power.
As usual, $\Fq$ and $\Fqt$ denote the finite fields with $q$ and $q^2$ elements respectively.
For $m \ge 0$, we denote by $\PP^m$, the projective space of dimension $m$ over the algebraic closure $\overline{\FF}_q$,
while $\PP^m (\Fqt)$ will denote the set of all $\Fqt$-rational points on $\PP^m$.
%Similarly, $\AA^m$ and $\AA^m (\Fqt)$ will denote the affine space of dimension $m$ over $\overline{\FF}_q$ and $\Fqt$ respectively.
Further, for a homogeneous polynomial $F \in \Fqt[x_0, \dots , x_m]$, we denote by $V(F)$, the set of zeroes of $F$ in $\PP^m$ and by
$V(F)(\Fqt)$ the set of all $\Fqt$-rational points of $V(F)$.
By an algebraic variety we will mean a set of zeroes of a certain family of polynomials in projective space.
In particular, an algebraic variety need not be irreducible.
We remark that, whenever we say that a variety is irreducible or nonsingular, we will mean that the variety is irreducible or nonsingular over
$\overline{\FF}_q$.
In this section, we recall the definition and various important properties of the Hermitian varieties.
We will indicate precise references in the text, but generally speaking, the results in this section come from \cite{BC} and \cite{C}.

\begin{definition}
For an $(m+1) \times (m+1)$ matrix $A = (a_{ij})$, $0 \le i, j \le m$, with entries in $\Fqt$, we denote by $A^{(q)}$, the matrix whose
$(i, j)$-th entry is given by ${a_{ij}}^q$.
The matrix $A$ is said to be a \textit{Hermitian matrix} if $A \neq 0$ and $A^T = A^{(q)}$.

A \textit{Hermitian variety} of dimension $m-1$, denoted by $V_{m-1}$, is the set of zeroes of the polynomial $x^T A x^{(q)}$ inside $\PP^m$,
where $A$ is an $(m+1) \times (m+1)$ Hermitian matrix and $x = (x_0, \dots, x_m)^T$.
The Hermitian variety is said to be \textit{non-degenerate} if $\mathrm{rank} \ A = m$ and \textit{degenerate} otherwise.
\end{definition}

%It is a well-known fact that if the rank of a Hermitian matrix is $r$, then by a suitable change of coordinates, we can describe the corresponding Hermitian variety by the zero set of the polynomial
It was established in \cite[Equation (5.6)]{BC} that, by a suitable linear change of coordinate systems, one may represent a Hermitian variety of
rank $r$ and dimension $m-1$ as the set of solutions of the equation
\begin{equation}\label{herm}
x_0^{q+1} + x_1^{q+1} + \dots + x_{r-1}^{q+1} = 0.
\end{equation}
in $\PP^m$.
It thus follows easily that a Hermitian variety of rank $r$ is irreducible over the algebraic closure of $\Fq$ whenever $r \ge 3$.
Throughout this article we will restrict our attention to Hermitian curves and Hermitian surfaces, i.e. Hermitian varieties of dimensions $1$ and
$2$ respectively.
We begin by recalling the following result concerning the intersection of lines with Hermitian surfaces.

\begin{lemma}\cite[Section 7]{BC}\label{line}
Any line in $\PP^3$ defined over $\Fqt$ satisfies precisely one of the following.
\begin{enumerate}
\item[(i)]  The line intersects $V_2$ at exactly $1$ point. These lines are called \emph{tangent lines.}
\item[(ii)] The line intersects $V_2$ at exactly $q+1$ points. These lines are called \emph{secant lines:}
\item[(iii)] The line is contained in $V_2$. These lines are called \emph{generators}.
\end{enumerate}
\end{lemma}
%\noindent Reflecting these three possibilities, we give the following definition.
%\begin{definition}\label{lines} \normalfont
%Let $\ell$ be a line in $\PP^3(\Fqt)$.  The line $\ell$ is called
%\begin{enumerate}
%\item[(a)]  a \textit{tangent line} if it intersects $V_2$ at exactly $1$ point.
%\item[(b)]  a \textit{secant line} if it intersects $V_2$ at exactly $q+1$ points.
%\item[(c)] a \textit{generator} if it is contained in $V_2$.
%\end{enumerate}
%\end{definition}

We denote by $\J$ the set of all generators of $V_2$. We recall the following equalities from \cite[Section 10]{BC}:
\begin{equation}\label{number}
|V_2 (\Fqt)| = (q^3 +1) (q^2 + 1) \ \ \ \ \text{and} \ \ \ \ |\J| = (q^3 + 1)(q + 1).
\end{equation}

Next, we recall the number of points in planar sections of the Hermitian surfaces:

\begin{theorem}\cite[Section 10]{BC}\label{linear}
Let $V_2$ denote a non-degenerate Hermitian surface in $\PP^3$. Let $\Pi$ be any hyperplane in $\PP^3$ defined over $\Fqt$.
If $\Pi$ is a tangent to $V_2$ at some point $P \in V_2$, then $\Pi$ intersects $V_2$ at exactly $q+1$ generators, all passing through $P$.
Otherwise, $\Pi$ intersects $V_2$ at a non-degenerate Hermitian curve $V_1$. In particular,
\begin{equation*}
|(V_2\cap \Pi)(\Fqt)| =
 \begin{cases}
q^3 + q^2 + 1 \ \ \ \ \mathrm{if} \ \Pi \ \mathrm{is \ a \ tangent \ plane,} \\
q^3 + 1 \ \ \ \ \ \ \ \ \ \ \mathrm{if} \ \Pi \ \mathrm{is \ not \ a \ tangent \  plane}.
\end{cases}
\end{equation*}
\end{theorem}

\begin{remark}\label{linetp} \normalfont
Let $\Pi$ be the tangent plane to $V_2$ at a point $P$. It follows that $\Pi$ contains $q+1$ generators and $q^2 - q$ tangent lines passing through $P$. All other lines defined over $\Fqt$ that are contained in $\Pi$ are secant lines. We refer to \cite[Section 10]{BC} for the proof of these results.
\end{remark}

Let $\ell$ be any line in $\PP^3$ defined over $\Fqt$. As introduced in \cite{BD}, by \textit{the book of planes around} $\ell$,
denoted by $\mathcal{B}(\ell)$, we mean the set of all planes in $\PP^3$ defined over $\Fqt$ that contain $\ell$.
%This set of planes is also called the pencil or sheaf of planes with axis $\ell$.
We note that, for any line $\ell$ in $\PP^3$ defined over $\Fqt$, the corresponding book has cardinality $q^2 + 1$.

\begin{proposition}\label{descbook}
Let $\ell$ be a line in $\PP^3$ defined over $\Fqt$ and $\mathcal{B}(\ell)$ be the book of planes around $\ell$.
\begin{enumerate}
\item[(a)] \cite[Lemma 5.2.3]{C} If $\ell$ is a generator, then every plane in $\mathcal{B}(\ell)$ is tangent to $V_2$ at some point of $\ell$.
\item[(b)] \cite[Lemma 5.2.6]{C} If $\ell$ is a tangent line, then exactly one plane in $\mathcal{B}(\ell)$ is tangent to $V_2$ at the point where $\ell$ meets $V_2$.
\item[(c)] \cite[Lemma 5.2.5]{C} If $\ell$ is a secant line, intersecting $V_2$ at $q+1$ points $P_0, \dots, P_q$, then out of the $q^2 + 1$ planes in $\mathcal{B}(\ell)$, exactly $q+1$ are tangent to $V_2$ at some point distinct from $P_0, \dots, P_q$.
\end{enumerate}
\end{proposition}

%\begin{corollary}\label{int}
%Let $\Pi_1$ and $\Pi_2$ be two planes in $\PP^3 (\Fqt)$ meeting at a line $\ell$. Then
%\begin{itemize}
%\item[(a)] If $\Pi_1, \Pi_2$ are both tangent to $V_2$, then $\ell$ is either a generator or a secant line.
%\item[(b)] If one of $\Pi_1, \Pi_2$ is not a tangent plane, then $\ell$ is either a secant or a tangent line.
%%\item[(b)] If $\Pi_1$ is a tangent plane but $\Pi_2$ is not, then $\ell$ is either a secant or a tangent line.
%%\item[(c)] If $\Pi_1, \Pi_2$ are planes that are not tangent to $V_2$, then $\ell$ is either a secant  or a tangent line.
%\end{itemize}
%\end{corollary}

%\begin{proof}
%Part (a) follows trivially by noticing that  the book of a tangent line contains exactly one tangent plane (see Prop. \ref{descbook} (b)), while part (b) is an easy consequences of the fact that all the planes in the book of a generator are tangent planes (see Prop. \ref{descbook} (a)).
%\end{proof}
%Based on the above combinatorial structure, S{\o}rensen considered \cite{SoT} the following arrangement of planes that attains the conjectured upper bound. We include a proof for the convenience of the reader.

We conclude the section by giving the following proposition which guarantees that the upper bound in Conjecture \ref{conj:soerensen} is attained by a union of $d$ hyperplanes.
This was first observed in \cite{SoT}, alternatively see \cite[Proposition 2.8]{BD}.

\begin{proposition}\label{attained}
Assume that $d \le q+1.$ Then there exist $d$ distinct planes $\Pi_1, \dots, \Pi_d$ be $d$ that are tangent to $V_2$ and contain a common secant line.
Moreover,
$$|(\Pi_1 \cup \dots \cup \Pi_d)(\Fqt) \cap V_2(\Fqt)| = d(q^3 + q^2 - q) + q + 1.$$
Consequently, there exists a homogeneous polynomial $F \in \Fqt[x_0, x_1, x_2, x_3]$ of degree $d$ such that
$V(F) = \Pi_1 \cup \dots \cup \Pi_d$.
\end{proposition}

Note that in light of part (c) of Proposition \ref{descbook}, the assumption that $d \le q+1$ is essential.

\section{Reduction to the union-of-lines case}\label{red}

%Lines contained in $V_2$ are called generators. Let $J$ be the set of generators of $V_2$. It is well known that $|J|=(q^3+1)(q+1).$ Moreover if two generators intersect in a point $P$, then they are contained in the tangent plane of $P$ to $V_2$. The intersection of this tangent plane and $V_2$ consists of $q+1$ distinct generators, all intersecting each other in the point $P$. In this way one establishes a bijection between points of $V_2$ and tangent planes of $V_2$.
In this section we begin our preparations to prove Conjecture \ref{conj:soerensen}.
In the following definition, we introduce some notation that will be used in the remainder of the article.

\begin{definition}\label{def:xdelta}
Let $F \in \F_{q^2}[x_0,x_1,x_2,x_3]$ be a homogeneous polynomial of degree $d$ and assume that $V(F)$ does not contain $V_2$ as irreducible component.
We define $$\X:= V(F) \cap V_2, \ \J_F := \{\ell \in \J \mid \ell \subset \X\}, \text{ and } \ \delta := d(q+1) - |\J_F|.$$
Further we write $\X = \LL \cup \X'$, where $\LL := \bigcup_{\ell \in \J_F} \ell$ and $\X'$ is a curve containing no lines defined over $\Fqt$.
\end{definition}

The condition that $V(F)$ does not contain $V_2$ is automatically true if $d \le q$, since $V_2$ has degree $q+1$.
If $V(F)$ does not contain $V_2$, then $\X$ is a complete intersection and $\deg \X = d(q+1)$, where $\deg \X$ denotes the (scheme theoretic)
degree of $\X$.
In particular we see that $\delta \ge 0$.
Note that $\deg \X' \le \deg \X -|\J_F| =\delta$.
%Therefore it is clear that there exists a nonnegative integer $\delta$ such that $V(F) \cap V_2$ is the union of $d(q+1)-\delta$ distinct lines (generators) and possibly a line-free (possibly reducible) curve of degree at most $\delta.$ We denote the subset of $J$ consisting of the $d(q+1)-\delta$ lines occurring in $V(F) \cap V_2$ by $J_F.$ We start by determining a lower bound on the degree of the component $\chi$.

The main goal of this section is to show that the inequality of the
Conjecture \ref{conj:soerensen} holds if $\X' (\Fqt) \neq \emptyset$.
We will in fact show that a stronger upper bound holds, even if no assumption on the degree $d$ is made.
We will assume throughout and in the remainder of the article as well,
that $V_2$ is not a component of $V(F),$ without stating this assumption explicitly in all theorems.

%This will reduce the proof of Conjecture \ref{conj:soerensen} to the case where $\X$ is a union of lines defined over $\Fqt$, at least as far as the $\Fqt$-rational points are concerned.
We begin with the following observation:

\begin{lemma}\label{lem:degbound}
Let $\C \subset V_2$ be a curve containing no lines defined over $\Fqt$. If $\deg \C < q+1$, then $\C$ contains no $\F_{q^2}$-rational points.
\end{lemma}
\begin{proof}
Let us assume that $\C(\F_{q^2})$ is not empty and fix a point $P \in \C(\F_{q^2})$.
We denote by $\Pi$ the tangent plane to $V_2$ at $P$.
Theorem \ref{linear} implies that $V_2 \cap \Pi = \bigcup_{i=1}^{q+1}\ell_i$, where the $\ell_i$ are the generators passing through $P$.
Therefore the intersection multiplicity of $\C$ and $\Pi$ at $P$ in $\PP^3$ satisfies:
$$\mathsf{i}(\C,\Pi;P)_{\mathbb{P}^3}=\mathsf{i}(\C,\Pi\cap V_2;P)_{V_2}=\sum_{i=1}^{q+1}\mathsf{i}(\C,\ell_i;P)_{V_2} \ge q+1.$$
In the first equality, we used the projection formula in intersection theory, see for example \cite[Appendix A, Section 1]{H}.
In the final inequality we used that for each line $\ell_i$ we have $\mathsf{i}(\C,\ell_i;P)_{V_2} \ge 1$, since $\C$ has none of the $\ell_i$ as
component.
On the other hand, again using that $\C$ contains none of the $\ell_i$, we have $\mathsf{i}(\C,\Pi;P)_{\mathbb{P}^3} \le \deg \C.$
This concludes the proof.
\end{proof}

%This lemma implies that when counting the number of rational points in $V(F) \cap V_2$, we may either assume that it is a union of generators, or that it contains a line-free component of degree at least $q+1.$

The following Theorem gives an upper bound on the number of $\Fqt$-rational points on $V(F) \cap V_2$ in terms of the number of generators
contained in the surface $V(F)$.

\begin{theorem}\label{thm:bound}
Let $\X$ and $\delta$ be as in Definition \ref{def:xdelta}. We have
%$$|V(F) \cap V_2| \le \frac{(q^3+1)(q+1)d+(d(q+1)-\delta)(q^2+1-d)}{q+1}.$$
$$|\X (\Fqt)| \le d(q^3+q^2-d+2)-\frac{\delta}{q+1}(q^2-d+1).$$
\end{theorem}
\begin{proof}
As before, $\J$ denotes the set of all generators of $V_2$.
Consider the set $\mathcal M:=\{(P,\ell) \in V(F)(\Fqt) \times \J \, : \, P \in \ell\}$ and the natural projection maps
$p_1:\mathcal M \rightarrow V(F)(\Fqt)$ and $p_2:\mathcal M \rightarrow \J.$
Then $\im(p_1)=\X(\Fqt)$.
Moreover, as observed in Section \ref{prel}, any $\Fqt$-rational point  $P \in V_2$ lies on exactly $q+1$ generators.
This implies that the fibre of any point in $V(F)(\Fqt)$ with respect to $p_1$ has $q+1$ elements. Furthermore, $|p_2^{-1}(\ell)|=q^2+1$ if $\ell \in \J_F$, while $|p_2^{-1}(\ell)| \le d$ if $\ell \in \J \setminus \J_F.$ Thus,
\begin{align*}\label{eq:doublecount}
|\X(\Fqt)|(q+1)=  |\mathcal M| & \le (q^2+1)|\J_F|+d|\J \setminus \J_F|\\
                    = & (d(q+1)-\delta)(q^2+1)+d((q^3-d+1)(q+1)+\delta)\\
                    = & d(q+1)(q^3+q^2-d+2)-\delta(q^2-d+1).
\end{align*}
The assertion now follows.
\end{proof}

%If the component $\chi$ mentioned in the theorem exists, then Lemma \ref{lem:degbound} implies that its irreducible components either do no have any $\Fqt$-rational points, or that its degree is at least $q+1$ and hence that $\delta \ge q+1$. Therefore we obtain the following.

As an immediate corollary of Lemma \ref{lem:degbound} and Theorem \ref{thm:bound} we obtain that the upper bound in the S{\o}renses's
conjecture is valid if $\X'$ contains an $\Fqt$-rational point:

\begin{corollary}\label{cor:bound0}
Let $\X, \X'$ and $\delta$ be as in Definition \ref{def:xdelta}. If $\X'(\Fqt) \neq \emptyset$, then
$$|\X (\Fqt)| \le dq^3+(d-1)q^2+1-(d-1)(d-2).$$
\end{corollary}
\begin{proof}
Since $\X'$ is a curve contained in $V_2$ and $\X' (\Fqt) \neq \emptyset$, Lemma \ref{lem:degbound} implies that
$\delta  \ge \deg \X' \ge q+1.$ Using Theorem \ref{thm:bound} we obtain the result.
%$|\X (\Fqt)| \le dq^3+(d-1)q^2+1-(d-1)(d-2).$
\end{proof}

\begin{remark}\label{rem:one}\normalfont
Note that for any $d \ge 1$ the upper bound derived in Corollary \ref{cor:bound0} is significantly stronger than the upper bound in
Conjecture \ref{conj:soerensen}.
This, in particular, implies that if a surface $V(F)$ of degree $d$, as always not containing $V_2$, attains the upper bound in S\o rensen's Conjecture,
then $V(F) \cap V_2$ is a union of at most $d(q+1)$ lines defined over $\Fqt$ and a curve without any $\Fqt$-rational points.
\end{remark}

\section{The case $V(F)$ does not contain a tangent plane of $V_2$}

In the previous section, we have already proved that the Conjecture \ref{conj:soerensen} is true if $\X' (\Fqt) \neq \emptyset$.
In this section, we restrict our attention to the case where $\X' (\Fqt) = \emptyset$, i.e. the case where all the $\Fqt$-rational points lie on
the $\Fqt$-linear components of $\X$.
Hence we will analyze the maximum number of $\Fqt$-rational points on various arrangements of lines defined over $\Fqt$.
The following definition will be crucial.
%By Corollary \ref{cor:bound0} we have a good upper bound on $|V(F) \cap V_2|$, if $V(F) \cap V_2$ contains at least one line-free curve with an $\Fqt$-rational point. In this section, we assume that such a component does not exist, that is to say $V(F) \cap V_2 = \bigcup_{\ell \in J_F} \ell \cup \chi$, with $\chi$ a line-free curve without $\Fqt$-rational points. Then $|V(F) \cap V_2|=|\bigcup_{\ell \in J_F} \ell|$ and we are led to study algebraic sets that are the union of at most $d(q+1)$ generators of $V_2$. %As before, we write $J$ (resp. $J_F$) for the set of all generators on $V_2$ (resp. on $V(F) \cap V_2$).
%%Now suppose that $V(F) \cap V_2 = \bigcup_{i=1}^{d(q+1)-\delta} \ell_i,$ where the $\ell_i$ are mutually distinct generators of $V_2$. As before, we write $J_F:=\{\ell_1,\dots,\ell_{d(q+1)-\delta}\}.$
\begin{definition}\normalfont
For $\ell \in \J_F$, we define
$$\T(\ell):=\{m \in \J_F \mid m \neq \ell, \ \ell \cap m \neq \emptyset\}  \ \text{ and } \ X:= \min_{\ell \in \J_F}| \T(\ell)|.$$
Further, for $\Pi \in \B(\ell),$ we define
$$\T_\Pi (\ell) := \{m \in \T(\ell) \mid m \subset \Pi\} \ \text{ and } \ a_{\Pi, \ell}:=|\T_\Pi (\ell)|.$$
\end{definition}

Note that $X \le |\J_F|-1 \le d(q+1)-1.$ Further, for any $\ell \in \J_F,$ we have
\begin{equation}\label{eq:sumofaPi}
\T (\ell) = \bigsqcup_{\Pi \in \B(\ell)} \T_\Pi (\ell) \ \ \text{and consequently,} \ \  |\T(\ell)| = \sum_{\Pi \in \B(\ell)} a_{\Pi, \ell}.
\end{equation}
Moreover, since $\ell \subset V(F) \cap \Pi$ and $\deg(V(F) \cap \Pi) =d$ if $\Pi \not\subset \V(F)$, we have:
\begin{equation}\label{eq:rangeaPi}
0 \le a_ {\Pi, \ell} \le d-1, \ \text{if} \ \Pi \not\subset \V(F).
\end{equation}
With this in place, we will derive several upper bounds on the cardinality of
$\X(\Fqt).$

\begin{theorem}\label{bound1}
%Let $|J_F|= d(q+1)-\delta$ for some nonnegative integer $\delta.$
Suppose that $V(F)$ does not contain a tangent plane of $V_2$.
If $\X' (\Fqt) = \emptyset$, then
$$|\X (\Fqt)| \le q^2+1+(d-1)(q^3+q)+(q^2-q)X.$$
In particular, if $X \le q+d-1$ and $\X' (\Fqt) = \emptyset$, then $|\X (\Fqt)| \le dq^3 + (d-1)q^2 + 1.$
\end{theorem}

\begin{proof}
If $\J_F=\emptyset$, then $|\X(\Fqt)|=|\X'(\Fqt)|=0$ and the theorem follows.
Otherwise, choose $\ell \in \J_F$ such that $|\T(\ell)|=X$ and fix
$\Pi \in \B(\ell)$.
Further define
$$W_1:=(V(F) \cap \Pi)\setminus \left(\ell \cup \bigcup_{m \in \T_{\Pi} (\ell)}m\right) \, \text{ and } \, W_2:=(V_2 \cap \Pi)\setminus \left(\ell \cup \bigcup_{m \in \T_{\Pi} (\ell)}m\right).$$
Since $\Pi$ contains the generator $\ell$, it is a tangent plane to $V_2$ and by hypothesis $\Pi \not\subset V(F)$.
%Thus $F|_{\Pi}$ is a nonzero polynomial of degree $d$.
It follows therefore that
$$\deg ( W_1) \le d-1 - a_{\Pi, \ell} \quad \text{and} \quad  \deg (W_2)  \le q - a_{\Pi, \ell}.$$
Also, $W_1$ and $W_2$ have no common components by definition of $\J_f$, as all components of $V_2 \cap \Pi$ are generators.
Thus Bezout's theorem applies and yields $$(W_1 \cap W_2)(\Fqt) \le \deg(W_1)\deg (W_2) \le (d-1 - a_{\Pi, \ell})(q - a_{\Pi, \ell}).$$
Further, note that for each $m \in \T_{\Pi} (\ell)$, we have $|(m \cap (\Pi \setminus \ell))(\Fqt)| = q^2$.
Then
\begin{align*}
|\X (\Fqt) \cap (\Pi\setminus \ell)(\Fqt)| & \le  \sum_{m \in \T_{\Pi}(\ell)}|m \cap (\Pi \setminus \ell)|+|(W_1 \cap W_2)(\Fqt)|\\
   &\le a_{\Pi, \ell} q^2+ (d-1-a_{\Pi, \ell})(q-a_{\Pi, \ell}).
\end{align*}

%\begin{align*}
%& \ \ |\X (\Fqt) \cap (\Pi\setminus \ell)(\Fqt)| \\
%&\le \sum_{m \in \T_{\Pi}(\ell)}|m \cap (\Pi \setminus \ell)| + \left|\left(\V(F) \cap ((\Pi\setminus \ell) \setminus \bigcup_{m \in \T_{\Pi} (\ell)}m)\right) \cap \left(V_2  \cap ((\Pi\setminus \ell) \setminus \bigcup_{m \in \T_{\Pi} (\ell)}m)\right)\right|\\
%&\le a_{\Pi, \ell} q^2+ (d-1-a_{\Pi, \ell})(q-a_{\Pi, \ell}).
%\end{align*}
%To prove the above inequality, we first note that for each $m \in \T_{\Pi} (\ell)$, we have $|m \cap (\Pi \setminus \ell)| = q^2$.
%Further, since $\Pi$ contains the generator $\ell$, it is a tangent plane to $V_2$ and by hypothesis $\Pi \not\subset V(F)$. Thus $F|_{\Pi}$ is a nonzero polynomial of degree $d$. It follows easily that
%$$\deg \left( V(F) \cap \big((\Pi\setminus \ell) \setminus \bigcup_{m \in \T_{\Pi} (\ell)}m \big)\right)= d-1 - a_{\Pi, \ell} \ \text{and} \  \deg \left(V_2  \cap \big((\Pi\setminus \ell) \setminus \bigcup_{m \in \T_{\Pi} (\ell)}m \big) \right) = q - a_{\Pi, \ell}.$$
%Also, the two varieties mentioned above do not have any common component as all the components of $V_2 \cap \Pi$ are generators. Thus Bezout's Theorem applies and yields the above inequality.
We obtain:
\begin{align*}
|(\X (\Fqt) \setminus \ell((\Fqt))| & \le \sum_{\Pi \in \B(\ell)} |\X (\Fqt) \cap (\Pi\setminus \ell)(\Fqt)| \\
& \le \sum_{\Pi \in \B(\ell)} a_{\Pi, \ell} q^2+ (d-1-a_{\Pi, \ell})(q-a_{\Pi, \ell})  \\
& = \sum_{\Pi \in \B(\ell)} (d-1)q+a_{\Pi, \ell}(q^2-q-d+1)+a_{\Pi, \ell}^2 \\
&  = (d-1)q(q^2+1)+X(q^2-q-d+1)+\sum_{\Pi \in \B(\ell)} a_{\Pi, \ell}^2.
\end{align*}
The last equality follows from equation \eqref{eq:sumofaPi}. Furthermore,
using equations \eqref{eq:sumofaPi} and \eqref{eq:rangeaPi}, we obtain
$$\sum_{\Pi \in \B(\ell)} a_{\Pi, \ell}^2 \le \sum_{\Pi \in \B(\ell)} (d-1)a_{\Pi, \ell}=(d-1) |\T (\ell)| = (d-1)X.$$
This implies,
\begin{align*}
|\X (\Fqt)|  & \le |\ell(\Fqt)|+(d-1)q(q^2+1)+X(q^2-q-d+1)+(d-1)X\\
             &    = q^2+1+(d-1)(q^3+q)+X(q^2-q).
\end{align*}
This completes the proof of the first assertion. The second assertion now follows easily.
\end{proof}
%\begin{corollary}\label{cor:bound1}
%Suppose that $V(F)$ does not contain a tangent plane of $V_2$ and suppose that $X \le q+d-1.$ Then $|V(F) \cap V_2| \le dq^3+(d-1)q^2+1.$
%\end{corollary}
%\begin{definition} \normalfont
%For $\ell \in J_F$ and $P \in \ell$, define $r_P:=|\{m \in J_F \mid P \in m\}|.$
%\end{definition}
%The number $r_P$ is closely related to the previously defined quantity $a_\Pi$ as the following lemma shows.
The following lemma gives a somewhat different interpretation of the quantity $a_{\Pi, \ell}$ that will be useful later on.
\begin{lemma}\label{pip}
%Let $\ell \in \J_F$ and $P \in \ell.$ Further denote by $\Pi$ the tangent plane of $V_2$ at $P$. Then $r_P=a_\Pi+1.$
Let $P \in V_2$ and $r_P:=|\{m \in \J_F \mid P \in m\}|$. For any line $\ell \in \J_F$ containing $P$, we have $r_P = a_{\Pi_P, \ell} + 1$, where $\Pi_P$ is the tangent plane to $V_2$ at $P$.
\end{lemma}
\begin{proof}
First, we note that any line $m \in \J_F$ is a generator of $V_2$ and consequently any such line passing through $P$ lies in $\Pi_P$.
This in particular implies that $a_{\Pi_P, \ell}$ is well defined and that $\{m \in \J_F \mid P \in m\} \subseteq \T_{\Pi_P} (\ell) \cup \{ \ell \}$. Further, Theorem \ref{linear} proves the reverse inclusion and the lemma follows.
%
%
%
%First note that since $V_2$ is a smooth variety, the tangent plane of $V_2$ at $P$ is defined. Since $P \in \ell$ and $\ell \subset V(F)$ by assumption, we obtain that $\ell \subset \Pi$ (that is $\Pi \in \mathcal{B}(\ell)$). Conversely, given $\Pi \in \mathcal{B}(\ell),$ then $\Pi$ is a tangent plane of $V_2$ at a unique point on $\ell.$ Hence there is a one-to-one correspondence between points on $\ell$ and planes in $\mathcal{B}(\ell).$ Moreover, if $m \subset V(F)$ and $m \subset \Pi$, then $P \in m$, since all $q+1$ generators contained in $\Pi$ intersect each other at $P$. This implies that if $\Pi$ is the tangent plant of $V_2$ at $P$, then $\{m \in J_F \mid P \in m\}=\{m \in J_F \mid m \subset \Pi\}.$ The lemma now follows.
\end{proof}

%The number $r_P$ is equal to the multiplicity of $P$ in the $\bigcup_{\ell \in J_F} \ell$, that is to say, the number of components of $\bigcup_{\ell \in J_F} \ell$ containing $P$. This point of view gives rise to the following theorem.

\begin{theorem}\label{thm:bound2}
Suppose that $V(F)$ does not contain a tangent plane of $V_2$.
If $\X' (\Fqt) = \emptyset$, then
$$|\X (\Fqt)|  \le d(q+1)\left(q^2+1-\frac{X}{d}\right).$$
In particular, if $X \ge q+ d-1$ and $\X' (\Fqt) = \emptyset$, then $|\X (\Fqt)| \le dq^3 + (d-1)q^2 + 1$.
\end{theorem}
\begin{proof}
If $\X' (\Fqt) = \emptyset$ then it is clear that $\X (\Fqt) = \LL (\Fqt) = \bigcup_{\ell \in J_F} \ell (\Fqt)$.
From Definition \ref{def:xdelta}, we have $|\J_F|=d(q+1)-\delta$.
By Lemma \ref{pip}, for any $P \in \bigcup_{\ell \in \J_F} \ell,$ the multiplicity $r_P$ of $P$ in the variety $\bigcup_{\ell \in \J_F} \ell$, equals
$a_{\Pi_P, \ell}+1$, where $\Pi_P$ is the tangent plane of $V_2$ at $P$ and  $\ell \in \J_F$ is chosen such that $P \in \ell$.
Since $P$ lies on exactly $r_P=a_{\Pi_P, \ell}+1$ lines of $\bigcup_{\ell \in \J_F} \ell,$ we see that
\begin{align*}
\displaystyle{\sum_{P\in \bigcup_{\ell \in \J_F} \ell}} (r_P-1)= &
\sum_{\ell \in \J_F} \sum_{P \in \ell(\Fqt)} \frac{r_P-1}{r_P}=
\sum_{\ell \in \J_F} \sum_{P \in \ell(\Fqt)} \frac{a_{\Pi_P, \ell}}{a_{\Pi_P, \ell}+1}=
\sum_{\ell \in \J_F} \sum_{\Pi \in \B(\ell)} \frac{a_{\Pi, \ell}}{a_{\Pi, \ell}+1}.
\end{align*}
The last equality is obtained by using the fact that there is a one-one correspondence between the set of $\Fqt$-rational points on $\ell$ and the set
of planes in $\B (\ell)$ that are tangent to $V_2$.
This implies,
\begin{align*}
|\bigcup_{\ell \in \J_F} \ell (\Fqt)| = & (d(q+1)-\delta)(q^2+1)-\sum_{P\in \cup_{\ell \in \J_F} \ell} (r_P-1)\\
=&(d(q+1)-\delta)(q^2+1)-\sum_{\ell \in \J_F} \sum_{\Pi \in \B(\ell)} \frac{a_{\Pi, \ell}}{a_{\Pi, \ell}+1}\\
\le & (d(q+1)-\delta)(q^2+1)-\sum_{\ell \in \J_F} \sum_{\Pi \in \B(\ell)} \frac{a_{\Pi, \ell}}{d}\\
= &(d(q+1)-\delta)(q^2+1)-\sum_{\ell \in \J_F} \frac{|\T(\ell)|}{d}\\
\le & (d(q+1)-\delta)(q^2+1)-\sum_{\ell \in \J_F} \frac{X}{d}\\
= & (d(q+1)-\delta)\left(q^2+1-\frac{X}{d}\right).
\end{align*}
The first inequality follows from \eqref{eq:rangeaPi} while the second inequality follows from the definition of $X$.
Since $X \le d(q+1),$ we have $q^2+1-X/d >0$ and the first claim in the theorem follows.
The second claim follows from a straightforward computation.
\end{proof}

\begin{corollary}\label{cor:bound2}
If $V(F)$ does not contain a plane tangent to $V_2$, then $|\X (\Fqt)| \le dq^3+(d-1)q^2+1.$
\end{corollary}
\begin{proof}
The assertion follows from Corollary \ref{cor:bound0}, Theorem \ref{bound1} and Theorem \ref{thm:bound2}
\end{proof}

Note that by using the concepts from Definition \ref{def:xdelta}, we have used the assumption made throughout in this article that
$V(F)$ does not contain $V_2$.
If $d \le q$, the bound derived in Corollary \ref{cor:bound2} is strictly smaller than the upper bound conjectured in
Conjecture \ref{conj:soerensen}. Thus to prove the upper bound conjectured in Conjecture \ref{conj:soerensen} it now only needs to be shown in
case $\X'(\Fqt)=\emptyset$ and $V(F)$ contains a tangent plane of $V_2$.
Moreover, if $|(V(F) \cap V_2)(\Fqt)|=d(q^3+q^2-q)+q+1$, then $V(F)$ would have to contain a tangent plane.

\section{Proof of S\o rensen's conjecture}

In this section we present our main result which proves the Conjecture \ref{conj:soerensen}.
We retain the notations and assumptions that were introduced in the previous sections.

\begin{theorem}\label{notan}
If $V(F)$ is not a union of planes that are tangent to $V_2$ then $|\X(\Fqt)| \le dq^3+(d-1)q^2+1$.
\end{theorem}
\begin{proof}
We prove the theorem by induction on $d$. Theorem \ref{linear} proves the assertion for $d=1$.
Let $d>1$ and suppose that the theorem holds for every surface $V(G)$, where $\deg G=d-1.$ The proof is divided into several cases.

%\noindent
%%{\bf Case 1:} Suppose that $\X' (\Fqt) \neq \emptyset$. Then Corollary \ref{cor:bound0} applies and proves the assertion.
%\bigskip
\noindent
{\bf Case 1:} Suppose that $V(F)$ does not contain a plane tangent to $V_2.$
In this case the assertion is proved using Corollary \ref{cor:bound2}. %implies that $|V_2 \cap V(F)| \le dq^3+(d-1)q^2+1.$

\noindent
{\bf Case 2:} Suppose $V(F)$ contains a plane $\Pi = V(H)$ tangent to $V_2$, where $H \in \Fqt[x_0, x_1, x_2, x_3]$ is a homogeneous polynomial of degree one.
If $\X' (\Fqt) \neq \emptyset$ then the assertion is proved using Corollary \ref{cor:bound0}.  We may thus assume that $\X' (\Fqt) = \emptyset$.
Write $F=G \cdot H$ where $G \in \Fqt[x_0, x_1, x_2, x_3]$ and $\deg G = d-1$.
Note that $V(G)$ is not a union of tangent planes, for otherwise $V(F)$ would be a union of tangent planes.
From the induction hypothesis, we have $|(V(G) \cap V_2)(\Fqt)| \le (d-1)q^3+(d-2)q^2+1.$ %Since $V(F) \cap V_2$ does not contain a line-free curve having an $\Fqt$-rational point, the same is true for $V(G) \cap V_2.$
Since we assume $\X' (\Fqt) = \emptyset$, it is clear that any $\Fqt$-rational point of $V(G) \cap V_2$ lies on a line in $\J_G$.
This implies that either $(V(G) \cap V_2)(\Fqt) = \emptyset$ or that $V(G) \cap V_2$ contains a generator $\ell.$
In the first case, we have $$|(V(F) \cap V_2)(\Fqt)| \le |(V(G) \cap V_2)(\Fqt)|+|(V(H) \cap V_2)(\Fqt)|=q^3+q^2+1,$$
while in the second case, we have
\begin{align*}
|(V(F) \cap V_2)(\Fqt)| \le & |(V(G) \cap V_2)(\Fqt)|+|(\Pi \cap V_2)(\Fqt)|-|(\Pi\cap V(G) \cap V_2)(\Fqt)|\\
 \le & \left((d-1)q^3+(d-2)q^2+1\right)+\left(q^3+q^2+1\right)-|(\ell \cap \Pi)(\Fqt)|\\
 \le & dq^3+(d-1)q^2+1.
\end{align*}
The last inequality follows since $(\ell \cap \Pi)(\Fqt) \neq \emptyset$. This completes the proof.
\end{proof}
\begin{corollary}
Let $F \in \Fqt[x_0, x_1, x_2, x_3]$ be a homogeneous polynomial of degree $d \le q$.
We have
$$|(V(F) \cap V_2)(\Fqt)| \le d(q^3+q^2-q)+q+1.$$
Moreover, equality holds if and only if $V(F)$ is the union of $d$ tangent planes of $V_2$ intersecting in a common secant line.
\end{corollary}
\begin{proof}
If $V(F)$ is not a union of planes that are tangent to $V_2$, then, using $d \le q$ and Theorem \ref{notan}, we obtain that $|\X (\Fqt)| < d(q^3+q^2-q)+q+1.$
We may thus assume that $V(F)$ is a union of planes that are tangent to $V_2$.
We prove the result using induction on $d$. For $d=1$, the result is a consequence of Theorem \ref{linear}.
Now assume that $d>1$ and that the result holds for polynomials $G \in \Fqt[x_0, x_1, x_2, x_3]$ where $\deg  G = d-1$ and $V(G)$ is a union of planes
that are tangent to $V_2$.
Write $F=H_1\cdots H_d$. We may further assume that $V(H_1), \dots, V(H_d)$ are distinct.
Since two distinct planes that are tangent to $V_2$ intersect each other at a line which is  either a generator or a secant  of $V_2$,
we have $|V(H_1 \cap H_2)(\Fqt)| \ge q+1.$
Therefore,
\begin{align*}
  &|(V(F) \cap V_2)(\Fqt)|  \\
      &= |(V(H_2\cdots H_d) \cap V_2)(\Fqt)|+|(V(H_1)\cap V_2)(\Fqt)|-|(V(H_2\cdots H_d) \cap V(H_1) \cap V_2)(\Fqt)|\\
 &\le |(V(H_2\cdots H_d) \cap V_2)(\Fqt)|+|(V(H_1)\cap V_2)(\Fqt)|-|(V(H_2) \cap V(H_1) \cap V_2)(\Fqt)|\\
 &\le (d-1)(q^3+q^2-q)+q+1+q^3+q^2+1-(q+1)\\
  &=  d(q^3+q^2-q)+q+1.
\end{align*}
Equality holds throughout if and only if each of the tangent planes intersect in a common secant.
\end{proof}

This shows that S\o rensen's conjecture is valid.
For $d=q+1$, a similar statement is true as we see now.

\begin{corollary}
Suppose that $d=q+1$ and that $V(F) \neq V_2$. Then $|\X (\Fqt)| \le (q+1)(q^3+q^2-q)+q+1$.
Moreover, equality holds if and only if $V(F)$ is the union of $d$ tangent planes of $V_2$ intersecting in a common secant line.
\end{corollary}

\begin{remark}\label{rem:two}\normalfont
If $d=q+1$ and $V(F) \neq V_2$, the bound in Corollary \ref{cor:bound2} and the conjectured bound in Conjecture \ref{conj:soerensen} are the same.
For $q>2$, an example attaining this bound is given by $F=\alpha(x_0^{q+1}+x_1^{q+1})+x_2^{q+1}+x_3^{q+1}$ where $\alpha \in \Fq \setminus\{0,1\}$.
Note that $V(F)$ is irreducible, hence does not contain a tangent plane, while $V(F) \cap V_2$ consists of the $(q+1)^2$ generators defined by $x_0-\zeta_1 x_1=x_2-\zeta_2x_3=0$,
where $\zeta_1,\zeta_2 \in \{\alpha \in \Fqt \mid \alpha^{q+1}=-1\}$. While it is not true that $V(F)$ is the union of $q+1$ tangent plane
intersecting in a common secant line, showing that the second part of Conjecture \ref{conj:soerensen} is not true in its current form,
it is clear that $V(F) \cap V_2=V(x_0^{q+1}+x_1^{q+1}) \cap V_2$. Hence in this example, $\X$ can still be obtained
as the intersection of $V_2$ and the union of $q+1$ tangent planes intersecting in a common secant line.
Proving this in general as well as understanding what happens for $d >q+1$ would be natural open problems for further study.
\end{remark}

%We can probably show that the second and third largest values among constellations of tangent planes are $d(q^3+q-q)-(d-3)(q+1)+1$ and $d(q^3+q-q)-(d-3)(q+1)$ if $d \ge 3.$ For $d=2$ or values of $d$ smaller than around $q/2$ (for which $dq^3+(d-1)q^2+1 \le d(q^3+q-q)-(d-3)(q+1)+1$), these will be the second and third largest values, otherwise the second largest value is at most $dq^3+(d-1)q^2+1.$ For $d=2$, this is confirmed by the results of Edoukou \cite{E}. If $d=3$, we find that the second and third largest values are $3(q^3+q-q)+1$ and $d(q^3+q-q)$, which extends the result in the paper submitted to MMJ and in particular also deals with the case $q=3$, which remained open there.

\section*{Acknowledgment}
The first author is supported by The Danish Council for Independent Research (DFF-FNU) with the project \emph{Correcting on a Curve}, Grant No.~8021-00030B.
The second author is supported by a postdoctoral fellowship from DST-RCN grant INT/NOR/RCN/ICT/P-03/2018.
The third author is deeply grateful to DTU for their hospitality during his stay.

\end{document}